\documentclass[11pt]{article}
\usepackage{amsfonts,amsmath,amssymb}
\usepackage{amscd}
\usepackage{latexsym}
\usepackage[english]{babel}

\def\NZQ{\Bbb}               
\def\NN{{\NZQ N}}

\def\ZZ{{\NZQ Z}}

\def\FFF{{\NZQ F}}

\def\binom#1#2{{#1 \choose #2}}

\def\opn#1#2{\def#1{\operatorname{#2}}} 
\opn\socdeg{socdeg}
\newcommand\Mon{{\operatorname{Mon}}}

\opn\reg{reg}
\opn\supp{supp}
\opn\Soc{Soc}
\opn\Tor{Tor} \opn\Ext{Ext} \opn\End{End}

\newcommand\m{{\operatorname{m}}}

\newtheorem{Theorem}{Theorem}[section]
\newtheorem{Lemma}[Theorem]{Lemma}

\newtheorem{Proposition}[Theorem]{Proposition}
\newtheorem{Remark}[Theorem]{Remark}

\newtheorem{Example}[Theorem]{Example}

\newtheorem{Definition}[Theorem]{Definition}

\newtheorem{Characterization}[Theorem]{Characterization}
\newtheorem{Problem}[Theorem]{Problem}

\newenvironment{proof}[1][Proof]{\textbf{#1.} }{$\Box$\medskip\medskip}

\let\epsilon\varepsilon
\let\phi=\varphi
\def\qed{\ifhmode\textqed\fi
   \ifmmode\ifinner\quad\qedsymbol\else\dispqed\fi\fi}
\def\textqed{\unskip\nobreak\penalty50
    \hskip2em\hbox{}\nobreak\hfil\qedsymbol
    \parfillskip=0pt \finalhyphendemerits=0}
\def\dispqed{\rlap{\qquad\qedsymbol}}

\def\NZQ{\Bbb}               
\def\NN{{\NZQ N}}

\def\ZZ{{\NZQ Z}}
\opn\Gin{Gin}

\newcommand\lex{{\operatorname{lex}}}

\newcommand\C{{\operatorname{Corn}}}
\newcommand\Shad{{\operatorname{Shad}}}
\newcommand\LexShad{{\operatorname{LexShad}}}

\newcommand\col{{\operatorname{col}}}
\numberwithin{equation}{section}

\renewcommand\footnotemark{}

\begin{document}

\title{Extremal Betti numbers of graded modules$^1$}
\author{Marilena Crupi \\
\ \\
{\footnotesize Department of Mathematics and Computer Science, University of Messina}\\
{\footnotesize Viale Ferdinando Stagno d'Alcontres 31, 98166 Messina, Italy}\\
{\footnotesize e-mail: mcrupi@unime.it}
}
\date{}
\thanks{$^1$ \textbf{Journal of Pure and Applied Algebra, 2016}}
\maketitle
\bibliographystyle{plain}

\begin{abstract}
Let $S$ be a polynomial ring in $n$ variables over a field
$K$ of characteristic $0$. A numerical characterization of all possible extremal Betti numbers of any graded submodule of a finitely generated graded free $S$-module is given.
\makeatletter{\renewcommand*{\@makefnmark}{}
\makeatother}
\end{abstract}
\thispagestyle{empty}

\section{Introduction} Let $S = K[x_1,\dots,x_n]$ be the polynomial ring in $n$ variables over a field
$K$, and $F = \oplus_{i=1}^m Se_i$ ($m\ge 1$) a finitely generated graded free
$S$-module with basis $e_1, \dots, e_m$ in degrees $f_1, \ldots,
f_m$ such that $f_1 \leq f_2 \leq \cdots \leq f_m$. The main purpose of this paper is to characterize the possible extremal Betti numbers (values as well as positions) of a graded submodule $M$ of $F$. Precisely, the following question is analyzed.
\begin{Problem} \label{probl} 
Given two positive integers $n, r$,  $1\le r \le n-1$, $r$ pairs of positive integers $(k_1, \ell_1)$,  $\ldots$, $(k_r, \ell_r)$ such that $n-1 \ge k_1 >k_2 > \cdots >k_r\ge 1$ and $1\le \ell_1 < \ell_2 < \cdots < \ell_r$, and $r$ positive integers $a_1, \ldots, a_r$, under which conditions does there exist a graded submodule $M$ of a finitely generated graded free $S$-module such that $\beta_{k_1, k_1+\ell_1}(M) = a_1$, $\ldots$, $\beta_{k_r, k_r+\ell_r}(M) = a_r$ are its extremal Betti numbers?
 \end{Problem}
 
The \emph{extremal Betti numbers} of a graded $S$-module $M$, introduced by Bayer, Charalambous and Popescu in \cite{BCP}, can be seen as a refinement of the Castelnuovo-Mumford
regularity and of the projective dimension of the module $M$; also, both position and value of any extremal Betti number can be read off the local cohomology modules
$H_{m}^i(M)$, with $m=(x_1, \ldots, x_n)$.  Such special graded Betti numbers are the \emph{non-zero top left corners in a block of zeroes in the Macaulay} \cite{GDS} or \emph{CoCoA Betti diagram} \cite{CNR} of a graded free resolution of $M$. Assume $K$ is a field of characteristic 0. Under this hypothesis,  
the generic initial module $\Gin(M)$, with respect to the graded reverse lexicographic order, of a graded $S$-module $M$ is a strongly stable submodule \cite{KP}, and the extremal Betti numbers of $M$, as well as their positions, are preserved by passing from $M$ to $\Gin(M)$ \cite[Corollary 1]{BCP}. Hence, Problem \ref{probl} is equivalent to the characterization of the possible extremal Betti numbers of a strongly stable submodule of $F$. A numerical characterization of the possible extremal Betti numbers of a graded ideal $I$ of initial degree $\ge 3$ of a standard graded polynomial ring over a field of characteristic $0$ has been given in \cite{CU2}. Recently, Herzog, Sharifan and Varbaro \cite{HSV} have given an alternative characterization (different in nature) of the extremal Betti numbers of such an ideal $I$, including also the case in which $I$ contains quadrics. Furthermore, the behavior of such graded Betti numbers has been studied for the class of lexicographic submodules of a finitely generated graded free $S$-module in \cite{CU1,CU3}, and 
for some classes of squarefree strongly stable submodules of a finitely generated graded free $S$-module in \cite{CF, CF2}.

The plan of the paper is as follows. In Section~\ref{sec1}, some notions that  will be used throughout the
paper are recalled. In Section \ref{sec2}, the possible extremal Betti numbers of a strongly stable ideal of initial degree $2$ are determined (Propositions~\ref{pro:twopairs}, \ref{pro:seven}); the case when the strongly stable ideal has initial degree $\ge 3$ has been analyzed and solved in \cite{CU2}. In Section \ref{sec3}, the possible extremal Betti numbers  (positions as well as values) of a strongly stable submodule of a finitely generated graded free $S$-module are computed (Theorem \ref{thm:possible}). Following the approach  of \cite{CU2}, the characterization is obtained by a detailed description of certain sets of monomials.

\section{Preliminaries} \label{sec1}
Let us consider the polynomial ring $S=K[x_1,\ldots, x_n]$ as an $\NN$-graded ring where each $\deg x_i =1$, and let $F = \oplus_{i=1}^m Se_i$ ($m\ge 1$) be a finitely generated graded free
$S$-module with basis $e_1, \dots, e_m$, where $\deg(e_i) =f_i$ for each $i=1, \ldots,m$, with $f_1 \leq f_2 \leq \cdots \leq f_m$. The elements of the form $x^ae_i$, where $x^a = x_1^{a_1} x_2^{a_2} \dots x_n^{a_n}$ for $a= (a_1, \dots, a_n)\in
\NN_0^n $, are called \textit{monomials} of $F$, and $\deg(x^a e_i) =\deg(x^a) + \deg(e_i)$. In particular, if $F\simeq S^m$ and $e_i = (0, \ldots, 0, 1, 0, \ldots,0)$, where $1$ appears in the $i$-th place, we assume, as usual, $\deg(x^a e_i) = \deg(x^a)$, \emph{i.e.}, $\deg(e_i)= f_i=0$. 

A \textit{monomial submodule} $M$ of $F$ is a submodule generated by monomials, \textit{i.e.}, $M = \oplus_{i=1}^m I_i e_i$, where $I_i$ are the monomial ideals of $S$ generated by those monomials $u$
of $S$ such that $ue_i \in M$ \cite{Ei}; if $m=1$ and $f_1=0$, then a monomial submodule is a monomial ideal of $S$.
If $I$ is a monomial ideal of $S$, we denote by $G(I)$ the
unique minimal set of monomial generators of $I$, by $G(I)_{\ell}$
the set of monomials $u$ of $G(I)$ such that $\deg u = \ell$, and by
$G(I)_{> \ell}$ the set of monomials $u$ of $G(I)$ such that $\deg u
> \ell$. Moreover, for a monomial $1 \neq u \in S$, we set
 \[\supp(u)=\{i: x_i\,\, \textrm{divides}\,\, u\},\qquad \m(u) = \max \{i:i\in \supp(u)\},\]
and $\m(1) = 0$.
A monomial ideal $I$ of $S$ is called \textit{stable} if for all $u \in G(I)$ one has
$(x_j u)/x_{\m(u)} \in I$ for all $j < \m(u)$. $I$  is called \textit{strongly stable} if for all $u \in G(I)$ one has
$(x_j u)/x_i \in I$ for all $i \in \supp(u)$ and all $j < i$ \cite{EK}.

Following \cite{CU3}, we introduce the following definition.

\begin{Definition} \label{def:stable} A graded submodule $M\subseteq F$ is a (strongly) stable submodule if
$M = \oplus_{i = 1}^m I_i e_i$ and $I_i\subsetneq S$ is a (strongly) stable ideal of $S$, for any $i=1, \ldots, m$.
\end{Definition}

For any $\ZZ$-graded finitely generated $S$-module $M$, there is a minimal graded free $S$-resolution \cite{BH}
\[
\FFF. : 0 \rightarrow F_s \rightarrow \cdots \rightarrow F_1
\rightarrow F_0 \rightarrow M \rightarrow 0,
\]
where $F_i = \oplus_{j \in \ZZ}S(-j)^{\beta_{i,j}}$. The
integers $\beta_{i,j} = \beta_{i,j}(M) = \textrm{dim}_K \Tor_i(K,
M)_j $ are called the \emph{graded Betti numbers} of $M$.

If $M = \oplus_{i=1}^m I_ie_i$  is a stable submodule of $F$, then $\beta_{k,\,k+\ell}(M) = \beta_{k,\,k+\ell}(\oplus_{i=1}^m I_ie_i) =
\sum_{i=1}^m \beta_{k,\,{k+\ell - f_i} }(I_i)$, and 
we can use the \emph{Eliahou-Kervaire  formula} \cite{EK} for computing the graded
Betti numbers of $M$:
\begin{equation}\label{eq:EK}
\beta_{k, k+\ell}(M)= \sum_{i=1}^m \left[\sum_{u \in G(I_i)_{\ell-f_i}}\binom {\m(u)-1}{k}\right].
\end{equation}
\begin{Definition} {\em \cite{BCP}} \label{def:extr} A graded Betti number $\beta_{k,k+\ell}(M) \neq 0$ is called {\it
extremal} if $\beta_{i,\, i+j}(M) = 0$ for all $i \geq k$, $j \geq
\ell$, $(i, j) \neq (k, \ell)$.
\end{Definition}
The pair $(k, \ell)$ is called a \emph{corner} of $M$.

From (\ref{eq:EK}), one can deduce the following characterization of the \emph{extremal Betti numbers} of a stable submodule.
\begin{Characterization}\label{char:CU} \label{equiv} Let $M = \oplus_{i=1}^mI_ie_i$ be a stable submodule of $F$.
A graded Betti number $\beta_{k, \, k+ \ell}(M)$ is extremal if and only if \[k + 1 = \max\{\m(u)\, :\, ue_i \in G(M)_{\ell}\,,
i\in\{1, \ldots, m\}\},\] and $\m(u) \leq k$, for all $ue_i \in G(M)_j$ and all $j > \ell$.
\end{Characterization}
As a consequence of the above result, one obtains that \emph{if $\beta_{k, \, k+ \ell}(M)$ is an extremal Betti number of $M$, then
$\beta_{k, \, k+ \ell}(M) = \vert\{ue_i \in G(M)_{\ell}\,:\, \m(u) =
k+1,\, i\in\{1, \ldots, m\}\}\vert$}.

\vspace{0,5cm}
For a monomial submodule $M = \oplus_{i=1}^mI_ie_i$ of $F$, denote by $\mathcal D (M) = \{I_1, \ldots, I_m\}$  the set of monomial ideals appearing in its direct decomposition. 

\begin{Remark} \label{value} \em Let $M = \oplus_{i=1}^mI_ie_i$ be a stable
submodule of $F$. If $\beta_{k, \, k+ \ell}(M)$ is an extremal Betti
number of $M$, then there exist some stable ideals $I_{j_1},
\ldots, I_{j_t}$ $\in \mathcal D(M)$,  $1 \leq j_1 < j_2 < \cdots < j_t \leq
m$, $(k, \ell)$ being a corner of $I_{j_i}e_{j_i}$ ($i=1, \ldots t$). 
Therefore, $\beta_{k, \, k+ \ell}(M) = \sum_{i=1}^t \beta_{k, \, k+
\ell}(I_{j_i}e_{j_i}) = \sum_{i=1}^t \beta_{k, \, k+ \ell-f_{j_i}}(I_{j_i})$
and $1 \leq \beta_{k, \, k+ \ell-f_{j_i}}(I_{j_i}) \leq \binom{k + \ell -f_{j_i}-1}{\ell -f_{j_i}-1}$, for every $i=1, \ldots, t$; $\binom{k + \ell -f_{j_i}-1}{\ell -f_{j_i}-1}$ is the number of all monomials $u$ of $S$
of degree $\ell-f_{j_i}$ with $\m(u) =k+1$.
\end{Remark}

Finally, for later use, let us recall that for a graded ideal $I=\oplus_{j \geq 0}I_j$ of $S$, the \textit{initial degree} of $I$, denoted by $\alpha(I)$, is  the minimum $s$ such that $I_s\neq 0$.

\section{Strongly stable ideals of initial degree $2$} \label{sec2}
In this Section, we analyze the possible extremal Betti numbers of a strongly stable ideal $I$ of $S=K[x_1, \ldots, x_n]$ of initial degree $2$. The case $\alpha(I) \ge 3$ was examined and completely solved in \cite{CU2}. We will see that, if we consider Problem \ref{probl}, the characterization of the possible extremal Betti numbers of a strongly stable ideal $I$ with $\alpha(I)=2$ depends on the number $r$ of the fixed pairs of integers.

We start this Section with some comments. If $S$ is a polynomial ring in $2$ variables, then a strongly stable ideal can have at most one extremal Betti number. If $S= K[x_1,x_2,x_3]$ is a polynomial ring in $3$ variables, then there exist strongly stable ideals $I$ of initial degree $2$ in $S$ with $2$ extremal Betti numbers.
Indeed, the strongly stable ideals $I = (x_1^2, x_1x_2, x_1x_3, x_2^{\ell})$, $\ell >2$, have the pairs $(2,2), (1,\ell)$ as corners.
These remarks justify our next assumption on the number $n$ of the variables of the polynomial ring $S$.

Denote by $\Mon(S)$ the set of monomials of $S$ and by $\Mon_d(S)$ the set of monomials of degree $d$ of $S$. 
For $u, v \in \Mon_d(S)$, $u \geq_{\lex}v$, define the set 
\[\mathcal{L}(u, v)=\{z \in \Mon_d(S) : u \geq_{\lex} z\geq_{\lex} v\},\] where $\ge_{\lex}$ is the usual \emph{lexicographic} ordering on the polynomial ring $S$ with respect to the ordering of the variables $x_1> x_2 >\cdots > x_n$.

\begin{Proposition} \label{pro: n-2} Let $n\ge 4$. A strongly stable ideal $I$ of $S$ of initial degree $2$ with a corner in degree $2$ can have at most $n-2$ extremal Betti numbers.
\end{Proposition}
\begin{proof} If $I\subsetneq S$ is a strongly stable ideal such that $x_2^2\notin G(I)_2$, and $\ell >2$ is the lowest degree for which $G(I)_{\ell}\neq \emptyset$, then $x_2^{\ell}\in G(I)_{\ell}$, necessarily. This implies that there is no extremal Betti number determined by a monomial $u\in G(I)_{> \ell}$ with $\m(u)=2$. If such $\ell$ does not exist, then $I$ is generated in degree $2$. More precisely, $G(I) = \mathcal{L}\left(x_1^2, x_1x_{k+1}\right)$, with $1\le k < n$, whence $I$ has just only one extremal Betti number. Furthermore, one can observe that if $\ell$ exists and $x_2^{\ell}$ determines an extremal Betti number of $I$, then $I$ is generated only in degrees $2$ and $\ell$. More specifically, $G(I)_2 = \mathcal{L}\left(x_1^2, x_1x_{k+1}\right)$, with $1\le k < n$, and $G(I)_{\ell} = \{x_2^{\ell}\}$. Thus, the ideal $I$ can have at most two extremal Betti numbers.
\end{proof}

\begin{Proposition} \label{cor:n-2} Let $n\geq 4$. Given $n-2$ pairs of positive integers
\begin{equation} \label{pairs1}
(k_1, \ell_1), (k_2, \ell_2), \ldots, (k_{n-2}, \ell_{n-2}),
\end{equation}
with $n-1 \ge k_1 >k_2 > \cdots >k_{n-2}\ge 2$ and $\ell_1 = 2 < \ell_2 < \cdots < \ell_{n-2}$,
then there exists a strongly stable ideal $I$ of $S$ of initial degree $2$ and with the pairs in (\ref{pairs1}) as corners if and only if $k_1+1 = n$.
\end{Proposition}
\begin{proof}  It is enough to show that, given two positive integers $n, r$, $1\leq r \leq n-2$, $r$ pairs of integers
\begin{equation} \label{pairs2}
(k_1, \ell_1), (k_2, \ell_2), \ldots, (k_r, \ell_r)
\end{equation}
with $n-1 \ge k_1 >k_2 > \cdots >k_r\ge 2$ and $2=\ell_1 < \ell_2 < \cdots < \ell_r$,
then there exists a strongly stable ideal $I$ of $S$ with $\alpha(I)=2$ having the pairs in (\ref{pairs2}) as corners. 

Indeed, setting $s=\max\{i : i\leq k_i+1\}$, the required monomial ideal $I$ can be constructed as follows:
\begin{enumerate}
\item[-] $G(I)_2 = \mathcal{L}\left(x_1^2, x_1x_{k_1+1}\right)$;
\item[-] $G(I)_{\ell_i} = \mathcal{L}\left(x_2^{\ell_2-2}x_3^{\ell_3-\ell_2}\cdots x_i^{\ell_i-\ell_{i-1}+2}, x_2^{\ell_2-2}x_3^{\ell_3-\ell_2} \cdots x_i^{\ell_i-\ell_{i-1}+1}x_{k_i+1}\right)$, for $i=2, \ldots, s$;
\item[-] $G(I)_{\ell_i}  =  \left \{x_2^{\ell_2-2}x_3^{\ell_3-\ell_2}\cdots x_{k_i}^{\ell_{k_i}-\ell_{k_i-1}-1}x_{k_i+1}^{3+\ell_i-\ell_{k_i}}\right\}$, for $i=s+1, \ldots, r$.
\end{enumerate}
It is worth observing that these choices imply $\beta_{k_i, \,k_i+\ell_i}(I)=1$, for all $i$.
\end{proof}

The next example illustrates the above results.
\begin{Example}\em
Set $S=K[x_1, \ldots,x_8]$. We may construct the strongly stable ideal $I = (x_1^2, x_1x_2, x_1x_3, x_1x_4,$ $x_1x_5, x_1x_6,x_1x_7,$ $x_1x_8, x_2^4, x_2^3x_3,x_2^3x_4,x_2^3x_5,
x_2^3x_6,x_2^2x_3^4,x_2^2x_3^3x_4,x_2x_3^8)$ $\subsetneq S$ having the pairs $(7,2)$, $(5,4), (3,6), (2,9)$ as corners, as well as
the strongly stable ideal
$J = (x_1^2$, $x_1x_2$, $x_1x_3$, $x_1x_4$,  $x_1x_5$, $x_1x_6$, $x_1x_7$, $x_1x_8$, $x_2^4$, $x_2^3x_3$, $x_2^3x_4$, $x_2^3x_5$,
$x_2^3x_6$, $x_2^3x_7$, $x_2^2x_3^3$, $x_2^2x_3^2x_4$, $x_2^2x_3^2x_5$, $x_2^2x_3^2x_6$, $x_2^2x_3x_4^4$, $x_2^2x_3x_4^3x_5$,
$x_2^2x_4^7$, $x_2x_3^9)\subsetneq S$ having the pairs $(7,2)$, $(6,4)$, $(5,5)$, $(4,7)$, $(3,9)$, $(2,10)$ as corners.
\end{Example}

Now, let $\mathcal{M}$ be a set of monomials of degree $d$ of $S$. 
The set of monomials of degree $d+1$ of $S$
\[
\Shad(\mathcal{M})=\{x_iu : \mbox{$ \, u \in \mathcal{M}$, \,\,  $i=1, \ldots, n$}\}
\]
is called the \textit{shadow} of $\mathcal{M}$. We define the $i$-th \emph{shadow} recursively by $\Shad^i(\mathcal{M})=\Shad(\Shad^{i-1}(\mathcal{M}))$. 
Moreover, we denote by $\min(\mathcal{M})$ the smallest monomial of $\mathcal{M}$ with respect to the lexicographic ordering on $S$. Setting $w=\min(\mathcal{M})$, if $\ell >d$ is an integer, we define the following set of monomials of degree $\ell$ in $S$:
\[\LexShad^{\ell-d}(\mathcal{M})=\mathcal{L}(x_1^\ell, wx_n^{\ell-d});\] 
we call such a set the \emph{lexicographic shadow} of $\mathcal{M}$.

Finally, given two positive integers $k, d$, with  $1\le k <n$ and $d \geq 2$, we introduce the following sets of monomials:
\[A(k, d) = \{ u \in \Mon_d(S) : \m(u) = k+1\}, \,\, A(\leq k, d) = \{ u \in \Mon_d(S) : \m(u) \leq k+1\}.\]

Setting $A(k, d) = \{u_1, \ldots, u_q\}$, we can suppose, after a permutation of the indices, that
\begin{equation} \label{vert}
u_1>_{\lex} u_2 >_{\lex} \cdots   >_{\lex} u_q;
\end{equation}
moreover, for the $i$-th monomial $u$ of degree $d$ with $\m(u) = k + 1$, we mean the monomial of $A(k, d)$ that appears in the $i$-th position of (\ref{vert}), for $1\le i \le q$.

\vspace{0,5cm}
From now on, we assume $S=K[x_1, \ldots, x_n]$ endowed with the lexicographic order $ >_{\lex}$ induced by the ordering $x_1 > x_2 > \cdots > x_n$.
\begin{Proposition} \label{pro:twopairs} Let $n\ge 4$ be an integer. Let
$(k_1, \ell_1), (k_2, \ell_2)$ be two pairs of positive integers such that $n-1 \ge k_1 >k_2 \ge 2$ and
$2 = \ell_1 < \ell_2$,
and let $a_1, a_2$ be two positive integers. If $K$ is a field of characteristic $0$, then the following conditions are equivalent:
\begin{enumerate}
\item[\em (1)] there exists a graded ideal $J \subsetneq S$, with extremal Betti numbers
$\beta_{k_1, k_1+\ell_1}(J)$ $= a_1, \beta_{k_2, k_2+\ell_2}(J) = a_2$;
\item[\em (2)] there exists a strongly stable ideal $I \subsetneq S$, with extremal Betti numbers
$\beta_{k_1, k_1+\ell_1}(I) = a_1, \beta_{k_2, k_2+\ell_2}(I) = a_2$;
\item[\em (3)] the integers $a_i$ satisfy the conditions:
\[1 \leq a_i\leq \vert A_i \setminus \LexShad^{\ell_i-\ell_{i-1}}(A_{i-1})\vert, \  \  \ \textrm{for} \  \  i=1,2, \]
where $A_0=\emptyset$, $A_1 = \{u \in A(k_1, \ell_1=2) : u \geq_{\lex}x_{k_2}x_{k_1+1}\}$ and $A_2 = \{u \in A(k_2, \ell_2) : u \geq_{\lex}x_{k_2+1}^{\ell_2}\}$.
\end{enumerate}
\end{Proposition}
\begin{proof}  (1) $\Leftrightarrow$ (2). Since $\textrm{char}(K) = 0$, then the generic initial ideal $\Gin(J) = I$ is a strongly stable ideal and the assertion follows from \cite[Corollary 1]{BCP}.\\
(2) $\Rightarrow$ (3). First of all note that $A_2\setminus \LexShad^{\ell_2-2}(A_1)\neq \emptyset$. In fact, $A_2\setminus \LexShad^{\ell_2-2}(A_1) = \{x_{k_2+1}^{\ell_2}\}$. Suppose $a_1> \vert A_1\vert$. Therefore, the greatest monomial  $u\in G(I)_2$ with $\m(u) = k_1+1$ following $x_{k_2}x_{k_1+1}$  is $u=x_{k_2+1}x_{k_1+1}$. Since $I$ is a strongly stable ideal, then $x_{k_2+1}^2\in G(I)_2$. Hence, the corner $(k_2,\ell_2)$ does not exist. A contradiction. Finally, it is easy to check that $a_2 \le \vert A_2\setminus \LexShad^{\ell_2-2}(A_1) \vert$. Indeed, $x_{k_2+1}^{\ell_2}$ is the smallest monomial belonging to $A(k_2, \ell_2)$. \\
(3) $\Rightarrow$ (2). Denote by $m(a_i)$ the $a_i$-th monomial of $A_i\setminus \LexShad^{\ell_i-\ell_{i-1}}(A_{i-1})$,
$i=1, 2$. We construct a strongly stable ideal $I$ of $S$ in degrees $\ell_1=2 < \ell_2$, using the following \emph{criterion}: $G(I)_2 = \{v \in A(\leq k_1, 2): x_1^2 \ge_{\lex}v\geq_{\lex} m(a_1)\}$; $G(I)_{\ell_2} = \{z \in A(\leq k_2, \ell_2): m(a_1)x_n^{\ell_2-2} >_{\lex}z\geq_{\lex} m(a_2)\}$, where $m(a_1)x_n^{\ell_2-2}$ is the smallest monomial belonging to $\Shad^{\ell_2-2}(G(I)_2)$. 
\end{proof}

\begin{Proposition} \label{pro:seven} 
Given two positive integers $n,r$ ($n\ge 5$) such that $2< r \leq n-2$,  let $(k_1, \ell_1), (k_2, \ell_2),$\, $ \ldots, (k_r, \ell_r)$ be $r$ pairs of positive integers with $n-1 \ge k_1 >k_2 > \cdots >k_r\ge 2$ and 
$2 = \ell_1 < \ell_2 < \cdots < \ell_r$,
and let $a_1, a_2, \ldots, a_r$ be $r$ positive integers.

If $K$ is a field of characteristic $0$, then the following conditions are equivalent:
\begin{enumerate}
\item[\em (1)] there exists a graded ideal $J \subsetneq S$, with extremal Betti numbers
$\beta_{k_i, k_i+\ell_i}(J) = a_i$, for $i=1,\ldots, r$;
\item[\em (2)] there exists a strongly stable ideal $I \subsetneq S$, with extremal Betti numbers
$\beta_{k_i, k_i+\ell_i}(I) = a_i$, for $i=1,\ldots, r$;
\item[\em (3)] set $t = \max\{i:\ell_i \le r-i\}$; the integers $a_i$ satisfy the conditions:
\[1 \leq a_i\leq \vert A_i \setminus \LexShad^{\ell_i-\ell_{i-1}}(A_{i-1})\vert, \  \  \ \textrm{for} \  \  i=1,\ldots, r \]
where $A_0=\emptyset$ and
\begin{enumerate}
\item[\em (i)] $A_1 = \{u \in A(k_1, \ell_1) : u \geq_{\lex}x_{k_r-1}x_{k_1+1}\}$;
\item[\em (ii)] $A_i = \{u \in A(k_i, \ell_i) : u \geq_{\lex}x_{k_r}\cdots x_{k_{r-\ell_i+3}}x_{k_{r-\ell_i+2}-1}x_{k_i+1}\}$, for $i=2, \ldots, t$;
\item[\em (iii)] $A_i = \{u \in A(k_i, \ell_i) : u \geq_{\lex}x_{k_r}\cdots x_{k_{i+1}}x_{k_i+1}^{\ell_i-(r-i)} \}$, for $i=t+1
\ldots, r-1$;
\item[\em (iv)] $A_r = \{u \in A(k_r, \ell_r) : u \geq_{\lex}x_{k_r+1}^{\ell_r}\}$.
\end{enumerate}
\end{enumerate}
\end{Proposition}
\begin{proof} 
(1) $\Leftrightarrow$ (2). See Proposition \ref{pro:twopairs}, (1) $\Leftrightarrow$ (2).\\
(2) $\Rightarrow$ (3). At first, one can observe that $A_i\setminus \LexShad^{\ell_i-\ell_{i-1}}(A_{i-1})\neq \emptyset$, for all $i=1,\ldots, r$.
Indeed, if $\ell_2 \le r-2$, then $A_2\setminus \LexShad^{\ell_2-2}(A_1) =$ $\{u \in A(k_2, \ell_2) : x_{k_r}^{\ell_2-1}x_{k_2+1} \ge_\lex u \ge_\lex (\prod_{j=r}^{r-\ell_2+3}x_{k_j})x_{k_{r-\ell_2+2}-1}x_{k_2+1}\}$; if  $\ell_2 > r-2$, 
$A_2\setminus \LexShad^{\ell_2-2}(A_1) = \{u \in A(k_2, \ell_2) : x_{k_r}^{\ell_2-1}x_{k_2+1} \ge_\lex u
\ge_\lex (\prod_{j=r}^{3}x_{k_j})x_{k_2+1}^{\ell_2-(r-2)}\}$, where $x_{k_r}^{\ell_2-1}x_{k_2+1}$ is the greatest monomial of $A(k_2, \ell_2)$ that does not belong to $\LexShad^{\ell_2-2}(A_1)$.\\
Moreover,
$A_i\setminus \LexShad^{\ell_i-\ell_{i-1}}(A_{i-1}) = \{(\prod_{j=r}^{r-\ell_i+3}x_{k_j})x_{k_{r-\ell_i+2}-1}x_{k_i+1}\}$, for $i=3, \ldots, t$, and $A_i\setminus \LexShad^{\ell_i-\ell_{i-1}}(A_{i-1}) = \{(\prod_{j=r}^{i+1}x_{k_j})x_{k_i+1}^{\ell_i-(r-i)}\}$,
for $i= t+1, \ldots, r-1$. Finally, $A_r\setminus \LexShad^{\ell_r-\ell_{r-1}}(A_{r-1}) =\{x_{k_r+1}^{\ell_r}\}$.\\
Suppose $a_1> \vert A_1\vert$. The greatest monomial  $u\in G(I)_2$ with $\m(u) = k_1+1$ following $x_{k_r-1}x_{k_1+1}$  is $u=x_{k_r}x_{k_1+1}$, and consequently, since $I$ is a strongly stable ideal, then $x_{k_r+1}^{\ell_2}\in G(I)_{\ell_2}$. Hence, the corner $(k_r,\ell_r)$ does not exist. A contradiction. \\
Now, set $T=\{\ell_i \,:\, \ell_i \le r-i\}$, for $i\in \{2, \ldots, r\}$. We distinguish two cases: $T = \emptyset$ and  $T \neq \emptyset$.\\
(Case 1). Assume $T = \emptyset$. Let $ 2 \le i < r$ be the smallest integer such that $a_i > \vert A_i\setminus \LexShad^{\ell_i-\ell_{i-1}}(A_{i-1})\vert$.  Then $G(I)_{\ell_i}$ contains the monomial $x_{k_r}x_{k_{r-1}} \cdots x_{k_{i+2}} x_{k_{i+1}+1}^{\ell_i-(r-i)}x_{k_i+1}$, which is the greatest monomial of $A(k_i, \ell_i)$ following $x_{k_r}\cdots x_{k_{i+1}}x_{k_i+1}^{\ell_i-(r-i)} = \min(A_i \setminus \LexShad^{\ell_i-\ell_{i-1}}(A_{i-1}))$. Since $\ell_2\ge 3$, using the same arguments as in \cite[Theorem 3.1]{CU2}, one has that $G(I)_{\ell_{r-1}}$ contains the monomial $x_{k_r+1}^{\ell_{r-1}}$. Hence, $(k_r, \ell_r)$ is not a corner of $I$. A contradiction. \\
(Case 2). Assume $T \neq \emptyset$. Let $2\le i \leq t$ be the smallest integer such that $a_i > \vert
A_i \setminus \LexShad^{\ell_i-\ell_{i-1}}(A_{i-1})\vert$. Then
$G(I)_{\ell_i}$ contains the monomial
$x_{k_r} \dots x_{k_{r-\ell_i+3}} x_{k_{r-\ell_i+2}} x_{k_i+1}$, which is the greatest monomial of $A(k_i, \ell_i)$ following
$x_{k_r}\cdots x_{k_{r-\ell_i+3}}x_{k_{r-\ell_i+2}-1}x_{k_i+1}=\min (A_i \setminus \LexShad^{\ell_i-\ell_{i-1}}(A_{i-1}))$. Therefore, using again the same arguments as in \cite[Theorem 3.1]{CU2}, one can verify that the monomial  $x_{k_r} \dots x_{k_{r-\ell_i+3}}  x_{k_{r-\ell_i+2}+1}^{\ell_{t+1}-(\ell_i-1)} x_{k_{t+1}+1}$ belongs to $G(I)_{\ell_{t+1}}$, and since $x_{k_r} \dots x_{k_{r-\ell_i+3}}  x_{k_{r-\ell_i+2}+1}^{\ell_{t+1}-(\ell_i-1)} x_{k_{t+1}+1}  <_{\lex}  x_{k_r} \dots x_{k_{t+2}} x_{k_{t+1}}^{\ell_{t+1}-(r-t-1)}$, a contradiction follows.\\
(3) $\Rightarrow$ (2) The proof is quite similar to Proposition \ref{pro:twopairs}. Denote by $m(a_i)$ the $a_i$-th monomial of $A_i \setminus \LexShad^{\ell_i-\ell_{i-1}}(A_{i-1})$, $1 \le i\le r$. We construct a strongly stable ideal $I$ of $S$ in degrees $\ell_1=2 < \ell_2 < \cdots <\ell_r$, using the following \emph{criterion}: $G(I)_2 = \{v \in A(\leq k_1, \ell_1=2): x_1^2 \ge_{\lex}v\geq_{\lex} m(a_1)\}$;
$G(I)_{\ell_i} = \{z \in A(\leq k_i, \ell_i): m(a_{i-1})x_n^{\ell_i-\ell_{i-1}} >_{\lex}z\geq_{\lex} m(a_i)\}$, where
 $m(a_{i-1})x_n^{\ell_i-\ell_{i-1}}$ is the smallest monomial belonging to $\Shad^{\ell_i-\ell_{i-1}}(\Mon(I_{\ell_{i-1}}))$, for $i=2, \ldots, r$, where $\Mon(I_{\ell_{i-1}})$ is the set of the monomials of degree $\ell_{i-1}$ belonging to $I_{\ell_{i-1}}$.
\end{proof}
\begin{Remark} \label{value3} \em 
A strongly stable ideal  $I \subsetneq S$ with $\alpha(I)\ge 3$ can have at most $n-1$ extremal Betti numbers; the possible extremal Betti numbers of such an ideal (positions as well as values) have been characterized in 
\cite[Theorem 3.1]{CU2}. 
\end{Remark}

Next characterization covers both the case considered in \cite[Theorem 3.1]{CU2} and the one in Proposition \ref{pro:seven}.

\begin{Theorem} \label{pro:general} Given two positive integers $n, r$ such that $1 \le r \leq n-1$, let $(k_1, \ell_1), (k_2, \ell_2),$ $ \ldots, (k_r, \ell_r)$ be $r$ pairs of positive integers with $n-1 \ge k_1 >k_2 > \cdots >k_r\ge 1$ and 
$2 \le \ell_1 < \ell_2 < \cdots < \ell_r$,
and let $a_1, a_2, \ldots, a_r$ be $r$ positive integers.

If $K$ is a field of characteristic $0$, then the following conditions are equivalent:
\begin{enumerate}
\item[\em (1)] there exists a graded ideal $J \subsetneq S$, with extremal Betti numbers
$\beta_{k_i, k_i+\ell_i}(J) = a_i$, for $i=1,\ldots, r$;
\item[\em (2)] there exists a strongly stable ideal $I \subsetneq S$, with extremal Betti numbers
$\beta_{k_i, k_i+\ell_i}(I) = a_i$, for $i=1,\ldots, r$;
\item[\em (3)] set $t = \max\{i:\ell_i \le r-i\}$; the integers $a_i$ satisfy the conditions:
\[1 \leq a_i\leq \vert A_i \setminus \LexShad^{\ell_i-\ell_{i-1}}(A_{i-1})\vert, \  \  \ \textrm{for} \  \  i=1,\ldots, r, \]
where $A_0=\emptyset$ and
\begin{enumerate}
\item[\em (i)] $A_1 = \{u \in A(k_1, \ell_1) : u \geq_{\lex}x_{k_r-1}x_{k_1+1}\}$, whenever $\ell_1=2$;
\item[\em (ii)] $A_i = \{u \in A(k_i, \ell_i) : u \geq_{\lex}x_{k_r}\cdots x_{k_{r-\ell_i+3}}x_{k_{r-\ell_i+2}-1}x_{k_i+1}\}$, for $i=1,\ldots, t$, whenever $\ell_1\ge 3$, and for $i=2, \ldots, t$, whenever $\ell_1=2$;
\item[\em (iii)] $A_i = \{u \in A(k_i, \ell_i) : u \geq_{\lex}x_{k_r}\cdots x_{k_{i+1}}x_{k_i+1}^{\ell_i-(r-i)} \}$, for $i=t+1,\ldots, r-1$;
\item[\em (iv)] $A_r = \{u \in A(k_r, \ell_r) : u \geq_{\lex}x_{k_r+1}^{\ell_r}\}$, 
\end{enumerate}
\end{enumerate}
with $2< r\le n-2$ (it has to be $n\ge 5$), $k_r\ge 2$, whenever $\ell_1=2$, and $1\le r\le n-1$, $k_r\ge 1$, whenever $\ell_1\ge 3$.
\end{Theorem}

\section{Possible corners}\label{sec3}
In this Section, we determine the possible corners (values and positions) of a strongly stable submodule of the finitely generated graded free $S$-module $S^m$, $m \ge 1$, with standard basis $e_1, \ldots, e_m$, where $e_i$ is the $m$-tuple whose only non--zero entry is a $1$ in the $i$-th position, and such that $\deg(e_i)= 0$, for all $i$.

Let us introduce some notations.
For a finitely generated graded $S$-module $M$, we denote by $\C(M)$ the set of all its corners, \emph{i.e.},
\[
\C(M) = \{(k, \ell) \in \NN \times \NN:  \beta_{k, k+\ell}(M)\, \,  \textrm{is an extremal Betti number of}\,\, M\}.
\]

If $M=\oplus_{j=1}^mI_je_j$ is a monomial submodule of the finitely generated graded free
$S$-module $F = \oplus_{i=1}^m Se_i$, with basis $e_1, \dots, e_m$ where $\deg(e_i) = f_i$, for  each $i=1, \ldots, m$, with $f_1\le \cdots \le f_m$, we denote by $\C_M(I_je_j)$ the set of corners of $I_je_j$ that are also corners of $M$, \emph{i.e.}, $\C_M(I_je_j) = \C(I_je_j) \cap \C(M)$.

Morever, if $\mathcal{D}(M)$ is the set of the ideals appearing in the direct decomposition of $M$, we define the following set of ideals of $S$:
\[\C(\mathcal{D}(M)) = \{I_j \in  \mathcal{D}(M) \,:\, \C_M(I_je_j) \neq \emptyset, \,\, \textrm{for}\,\, j=1,\ldots, m\};\]
we call each ideal of $\C(\mathcal{D}(M))$ a \emph{corner ideal} of $M$.
\begin{Definition} \label{def:corn} Let $(k_1, \ldots, k_r)$ and $(\ell_1, \ldots, \ell_r)$ be two sequences of positive integers such that $n-1\ge k_1> k_2>\cdots >k_r \ge 1$ and $1\le \ell_1 < \ell_2 < \cdots < \ell_r$. The set $\mathcal{C} = \{(k_1, \ell_1), \ldots,(k_r, \ell_r)\}$ is called a corner sequence and $\ell_1, \ldots, \ell_r$ are called the \emph{corner degrees} of $\mathcal{C}$.
\end{Definition}

\begin{Remark} \em Definition \ref{def:corn} is motivated by the fact that for every finitely generated graded $S$-module $M$, $\C(M)$ defines a corner sequence.
\end{Remark}

\begin{Definition}
A totally ordered corner sequence 
$
\mathcal{C} = \{(k_1, \ell_1), \ldots, (k_r, \ell_r)\}
$ 
is a corner sequence such that
$(k_1, \ell_1)\succ (k_2, \ell_2)\succ  \cdots \succ (k_r, \ell_r)$
where
\[
(k_i, \ell_i)\succ (k_j, \ell_j)\quad  \hbox{if and only if}\quad
k_i\ge k_j \quad \hbox{and} \quad \ell_i \le \ell_j;
\]
we refer to $(k_i, \ell_i)$ as the $i$-th element of the ordered corner sequence. 
\end{Definition}

Let $M=\oplus_{j=1}^m I_je_j$ be a monomial submodule of $F$. One can observe that if $(k,\ell)\in \C(M)$ and $\beta_{k, k+\ell-f_j}(I_j)$ $\neq 0$, then $(k,\ell)\in \C_M(I_je_j)$. Hence, we define the following $m$-tuple of non--negative integers:
\[C_{(k, \ell)} = (\beta_{k, k+\ell-f_1}(I_1), \ldots, \beta_{k, k+\ell-f_m}(I_m));\]
we call such a sequence the $(k, \ell)$-\emph{sequence} of the module $M$.  
It is clear that $\beta_{k, k+\ell}(M) = \sum_{j=1}^m\beta_{k, k+\ell-f_j}(I_j)$.

Hence, if $\C(M)= \{(k_1, \ell_1), \ldots, (k_r, \ell_r)\}$ is a totally ordered corner sequence, 
one can associate to the module $M$ an $r \times m$ matrix whose $i$-th row is the $(k_i, \ell_i)$-sequence of $M$, $1\le i \le r$. We call such a matrix the \emph{corner matrix} of $M$, and we denote it by $C_M$.  The sum of the entries of the 
$i$--th row of $C_M$ equals the value of the extremal Betti number $\beta_{k_i, k_i+\ell_i}(M)$, for $i=1,\ldots, r$.\\

\begin{Example}  \em Let $S=K[x_1, \ldots, x_6]$. Consider the strongly stable submodule $M = I_1e_1\oplus I_2e_2 \oplus I_3e_3 \oplus I_4e_4\subsetneq S^3$,
where $I_1= (x_1^2, x_1x_2,x_1x_3,x_1x_4,x_1x_5, x_1x_6, x_2^3, x_2^2x_3, x_2^2x_4,x_2x_3^4)$, 
$I_2 = (x_1^2, x_1x_2,x_1x_3)$, $I_3=(x_1^3, x_1^2x_2,x_1^2x_3, x_1^2x_4)$, $I_4 = (x_1^2, x_1x_2,x_1x_3,x_1x_4,x_1x_5,  x_2^3, x_2^2x_3, $\newline $x_2^2x_4,x_2x_3^2, x_2x_3x_4, x_2x_4^2)$. 
By using the computer program CoCoA \cite{CNR}(see also \cite{GDS}), one has that $\C(M)= \{(5,2), (3,3), (2,5)\}$, 
$\C(I_1e_1)=\C_M(I_1e_1)=\C(M)$, $\C(I_2e_2)= \{(2,2)\}$, $\C_M(I_2e_2)= \emptyset$, $\C(I_3e_3)=\C_M(I_3e_3)=\{(3, 3)\}$, $\C(I_4e_4)$\qquad $= \{(4,2), (3,3)\}$, $\C_M(I_4e_4)=\{(3,3)\}$. Moreover, $\beta_{5, 5+2}(I_1e_1)=\beta_{3, 3+3}(I_1e_1)=\beta_{2, 2+5}(I_1e_1)=1$,  $\beta_{3, 3+3}(I_3e_3)=1$,  $\beta_{3, 3+3}(I_4e_4)=3$ and $\beta_{5, 5+2}(M)=1$, $\beta_{3, 3+3}(M)=5$, $\beta_{2, 2+5}(M)=1$.

Hence, $C_{(5, 2)} =(1,0,0,0)$, $C_{(3, 3)} = (1,0,1,3)$, $C_{(2, 5)} = (1,0,0,0)$,
and 
\[C_M= \left(
\begin{array}{llll}
1 & 0 & 0 & 0\\
1 & 0 & 1 & 3\\
1 & 0 & 0 & 0\\
\end{array}
\right).
\]

Note that the sum of the entries of each row of $C_M$ gives the value of the extremal Betti numbers $\beta_{5, 5+2}(M)$, $\beta_{3, 3+3}(M)$ and $\beta_{2, 2+5}(M)$, respectively. Furthermore, $\C(\mathcal{D}(M)) = \{I_1, I_3, I_4\}$.
\end{Example}

From now on, we deal with submodules of the finitely generated graded free $S$--module $S^m$, $m \ge1$, with standard basis $e_1, \ldots, e_m$, where $\deg(e_i) =0$, for all $i$. Moreover, in order to simplify the notations, if $M=\oplus_{j=1}^m I_je_j$ is a monomial submodule of $S^m$, we write $\C_M(I_je_j) = \C_M(I_j)$ and $\C(I_je_j) = \C(I_j)$.

\begin{Lemma} \label{lem:sub} Let $M=\oplus_{j=1}^m I_je_j$ be a strongly stable submodule of $S^m$, $m\ge 1$, with extremal Betti numbers $\beta_{k_1, k_1+\ell_1}(M) = a_1$, $\ldots$, $\beta_{k_r, k_r+\ell_r}(M) = a_r$,  $1\le r\le n-1$. Then there exists a strongly stable submodule $\widetilde M$ of $S^m$ such that
\begin{enumerate}
\item[\em(i)] $\C(\widetilde M) = \C(M)$;
\item[\em(ii)] $\beta_{k_i, k_i+\ell_i}(\widetilde M) = a_i$, for $i=1, \ldots, r$;
\item[\em(iii)] $\C(I) = \C_{\widetilde M}(I)$,  for any $I \in \C(\mathcal{D}(\widetilde M))$.
\end{enumerate}
\end{Lemma}
\begin{proof} Let $(k_i, \ell_i)$ be a corner of $M$ and let $C_{(k_i, \ell_i)} =
(\beta_{k_i, k_i +\ell_i}(I_1), \ldots, \beta_{k_i, k_i+\ell_i}(I_m))$ be the $(k_i, \ell_i)$-sequence of $M$, for $i=1, \ldots, r$. Setting $\beta_{k_i, k_i +\ell_i}(I_j)=a_{i,j}$, $j=1,\ldots,m$, one has $a_i= \sum_{j=1}^ma_{i,j}$, for $i=1, \ldots, r$.  Consider the corner matrix of $M$:
\begin{equation} 
\label{matrix}
C_M = \left( \begin{array}{cccc}
a_{1,1} & a_{1,2} & \cdots  & a_{1,m} \\
a_{2,1} & a_{2,2} & \cdots  & a_{2,m} \\
\cdots & \cdots & \cdots  & \cdots \\
a_{r,1} & a_{r,2} & \cdots & a_{r,m} 
\end{array} \right).
\end{equation} 
Setting 
\[\mathcal{C}_h = \{(k_i, \ell_i) \,:\, a_{i,h}\neq 0, \,\, \mbox{for $i=1, \ldots, r$}\}\, (1 \le h \le m),\]
let $I_h\in \C(\mathcal{D}(M))$ be the strongly stable ideal of $S$ such that $\C_M(I_h) = C_h$. One has that   
$\beta_{k_i, k_i +\ell_i}(I_h)=a_{i,h}$, for $i=1, \ldots, r$.\\
Now we prove that there exists a strongly stable ideal $\widetilde I_h$ ($1\le h \le m$) of $S$ such that $\C(\widetilde I_h) = C_h$ and $\beta_{k_i, k_i +\ell_i}(\widetilde I_h)=a_{i,h}$, for $i=1, \ldots, r$.
\par\noindent
If $\C(I_h)=\C_M(I_h)$, then $I_h$ is the ideal we are looking for. Assume $C_h \subsetneq \C(I_h)$. Thus,
$C_h$ is not a maximal corner sequence of $I_h$, and there exists at least one pair $(k, \ell)$ of positive integers which is a corner of $I_h$, but not of $M$. Hence, because of $\vert C_h\vert < \vert \C(I_h)\vert$, we are able to construct a strongly stable ideal $\widetilde I_h$ of $S$ generated in the corner degrees of $C_h$, with $\C(\widetilde I_h)=C_h$ and such that $\beta_{k_i, k_i +\ell_i}(\widetilde I_h)=a_{i,h}$, using the same criterion in Propositions \ref{pro:twopairs}, \ref{pro:seven} (see also \cite[Theorem 3.1]{CU2}).

Finally, the strongly stable submodule $\widetilde M$ of $S^m$ obtained from $M$, replacing every ideal $I_h\in \C(\mathcal{D}(M))$  with the ideal $\widetilde I_h$, and leaving unchanged the ideals $I_j\in \mathcal{D}(M)\setminus  \C(\mathcal{D}(M))$, is the desired strongly stable submodule of $S^m$.
\end{proof}

Now we are in position to give the main result of the paper.

\begin{Theorem} \label{thm:possible} Given two positive integers $n, r$ such that $1\leq r \le n-1$, $r$ pairs of positive integers $(k_1, \ell_1),\ldots, (k_r, \ell_r)$ with $n-1 \ge k_1 > k_2 > \cdots > k_r\ge 1$
and $2\le \ell_1 < \ell_2 < \cdots < \ell_r$, and $a_1, \ldots, a_r$ positive integers. Set $\mathcal{C} = \{(k_1, \ell_1),\ldots, (k_r, \ell_r)\}$. Then there exists a strongly stable submodule $M$ of $S^m$, $m\ge 1$, with extremal Betti numbers $\beta_{k_i, k_i+\ell_i}(M)=a_i$, for $i=1,\ldots, r$, if and only if
\begin{enumerate}
\item[\em (1)] for $m=1$, $k_1+1=n$, whenever $r=n-2$, $\ell_1=2$ and $k_r\ge 2$, or whenever $r=n-1$, $\ell_1\ge 3$ and $k_r\ge 1$;
\item[\em (2)] $1 \le a_i \le m\binom{k_i+\ell_i-1}{\ell_i-1}$;
\item[\em (3)] for all $i\in \{1,2,\ldots,r\}$, there exists  an $m$-tuple of non-negative integers $(a_{i,1}, a_{i,2}, \ldots,$\qquad $a_{i,m})$  such that
\begin{enumerate}
\item[\em (i)] $a_i= \sum_{h=1}^m a_{i,h}$; 
\item[ \em (ii)] for all $h\in \{1,2,\ldots,m\}$ for which the $r$-tuple $(a_{1,h}, a_{2,h}, \ldots, a_{r,h})$ has at least one non-zero entry,
if $\mathcal{C}_h$ is the corner subsequence of the corner sequence $\mathcal{C}$ consisting of all pairs $(k_i, \ell_i)$ such that $a_{i,h}\neq 0$, set
\begin{equation} 
\label{matrix}
\mathcal{A} = \left( \begin{array}{cccc}
a_{1,1} & a_{1,2} & \cdots  & a_{1,m} \\
a_{2,1} & a_{2,2} & \cdots  & a_{2,m} \\
\cdots & \cdots & \cdots  & \cdots \\
a_{r,1} & a_{r,2} & \cdots & a_{r,m} 
\end{array} \right),
\end{equation} 
one has
\begin{enumerate}
\item[{\rm (ii.1)}] $1 \le a_{i,h}\le \binom{k_i+\ell_i-1}{\ell_i-1}$, for $\vert \mathcal{C}_h \vert=1$,
or
\item[{\rm (ii.2)}] $1 \leq a_{i,h}\leq \vert A_i\setminus \LexShad^{\ell_i-\ell_j}(A_j)\vert$, for  $\vert \mathcal{C}_h \vert >1$, where $1\le j\le i-1$, is the greatest integer such that $a_{j,h}\neq 0$,  and 
$A_i$ ($A_j$, respectively) are the sets of monomials of degree $\ell_i$ (degree $\ell_j$, respectively) defined in the statements of Proposition \ref{pro:twopairs} and Theorem \ref{pro:general}.  
\end{enumerate}

\end{enumerate}
\end{enumerate}
\end{Theorem}
\begin{proof} For $m=1$, the assertions (1)--(3) follow from Propositions \ref{cor:n-2}, \ref{pro:twopairs}, Theorem \ref{pro:general} and Remarks \ref{value}, \ref{value3}. Assume $m>1$. Let $M=\oplus_{j=1}^m I_je_j$ be a strongly stable submodule of $S^m$ with extremal Betti numbers $\beta_{k_i, k_i+\ell_i}(M)=a_i$, for $i=1,\ldots, r$. 
Condition (2) follows  from Remark \ref{value}. Furthermore, one can observe that $\beta_{k_i, k_i+\ell_i}(M)$ equals the value $m\binom{k_i+\ell_i-1}{\ell_i-1}$ if the corner $(k_i, \ell_i)$ is determined by every ideal $I_j\in \mathcal{D}(M)$ and 
$\beta_{k_i, k_i+\ell_i}(I_j) = \binom{k_i+\ell_i-1}{\ell_i-1}$, for each $j$.
Let $(k_i, \ell_i)\in \C(M)$ and consider the $(k_i, \ell_i)$-\emph{sequence} $C_{(k_i, \ell_i)}$ of $M$, $i=1, \ldots, r$. Setting $C_{(k_i, \ell_i)} = (a_{i,1}, a_{i,2}, \ldots, a_{i,m})$, one has $a_i= \sum_{h=1}^m a_{i,h}$. The non-zero components of $C_{(k_i, \ell_i)}$ are determined by the corner ideals $I_j$ of $M$, $1\le j\le m$, such that $(k_i, \ell_i) \in \C_M(I_j)$. From Lemma \ref{lem:sub}, we may assume that $\C_M(I_j) = \C(I_j)$. If  $I_h$ is such an ideal, then $\beta_{k_i, k_i+\ell_i}(I_h)=a_{i,h}$.
We distinguish two cases: $\vert \C_M(I_h)\vert =1$ and $\vert \C_M(I_h)\vert >1$.\\
(Case 1.) Let $\vert \C_M(I_h)\vert =1$. Then Condition (3),(ii.1) follows from Remark \ref{value}. \\
(Case 2.) Let $\vert \C_M(I_h)\vert >1$.  If $(k_j, \ell_j)$, with $1\le j \le i-1$, is the smallest element of $\C_M(I_h)$ (with respect to $\succ$) such that $(k_j, \ell_j)\succ (k_i, \ell_i)$,
then Condition (3), (ii.2) follows from Theorem \ref{pro:general} and Proposition \ref{pro:twopairs}.

Conversely, let $a_1, \ldots, a_r$ be positive integers satisfying Conditions (2) and (3). First of all note that Condition (1) assures the existence of a strongly stable ideal of $S$ of initial degree $2$ (of initial degree $\ge 3$, respectively) with $n-2$ extremal Betti numbers ($n-1$ extremal Betti numbers, respectively) in the given positions. 

Consider the $r\times m$ matrix defined in (\ref{matrix}).
Our purpose is to construct a strongly stable submodule $M$ of $S^m$ such that $C_M = \mathcal{A}$.

Denote by $\col^*(\mathcal{A})$ the set consisting of the columns of the matrix $\mathcal{A}$ having  at least one non-zero entry. We distinguish two cases: 
$\vert \col^*(\mathcal{A})\vert = 1$ and $\vert \col^*(\mathcal{A})\vert > 1$.\\
(Case 1.) Let $\vert \col^*(\mathcal{A})\vert=1$. Consider the monomial submodule $M=\oplus_{j=1}^m I_je_j$ of $S^m$ defined as follows: 
\begin{enumerate}
\item[-] $G(I_j) = \mathcal{L}(x_1^{\ell_1-1}, \ldots, x_1^{\ell_1-2}x_{k_r+1})$, for $j=1, \cdots, m-1$; 
\item[-] $I_m$ is the strongly stable ideal of $S$ generated in degrees $\ell_1 < \cdots < \ell_r$ with extremal Betti numbers $\beta_{k_i, k_i+\ell_i}(I_m)=a_i$, for $i=1, \ldots, r$, and constructed according to the criterion used in Propositions \ref{pro:twopairs}, \ref{pro:seven} (see also \cite[Theorem 3.1]{CU2}) .
\end{enumerate}
Since the ideals $I_j$, $j=1, \ldots, m-1$, do not give any contribution to the computation of the extremal Betti numbers of $M$ (Characterization \ref{char:CU}), it is clear that with these choices $M$ is a strongly stable submodule of $S^m$ with extremal Betti numbers $\beta_{k_i, k_i+\ell_i}(M)=a_i$, for $i=1, \ldots, r$.  \\
(Case 2.) Let $\vert \col^*(\mathcal{A})\vert > 1$.  
We distinguish two subcases: $\vert \col^*(\mathcal{A})\vert = m$; $\vert \col^*(\mathcal{A})\vert < m$.\\
(Subcase 2.1) Suppose $\vert \col^*(\mathcal{A})\vert = m$. Let $\col_h(\mathcal{A})= (a_{1,h}, a_{2,h},\ldots, a_{r,h})$ be the $h$-th column of $\mathcal{A}$, $1\le h \le m$, and set
\[\mathcal{C}_h = \{(k_i, \ell_i) \,:\, a_{i,h}\neq 0, \,\, \mbox{for $i=1, \ldots, r$}\}.\]
One can observe that Condition (3),(ii) assures the existence of a strongly stable ideal $I_h\subsetneq S$, for each $h=1, \ldots, m$, such that:
\begin{enumerate}
\item[-] $\C(I_h) = \mathcal{C}_h$;
\item[-] $I_h$ is generated in the corner degrees of $\mathcal{C}_h$;
\item[-] $\beta_{k_i, k_i+\ell_i}(I_h)=a_{i,h}$, for $i=1, \ldots, r$.
\end{enumerate} 
Indeed, because of the behavior of the integers $a_{i,h}$, such an ideal can be obtained by using the same criterion as in Proposition \ref{pro:twopairs} (see also Proposition \ref{pro:seven} and \cite[Theorem 3.1]{CU2}). Hence, $M= \oplus_{h=1}^m I_he_h$ is the desired strongly stable submodule of $S^m$.\\
(Subcase 2.2) Let $\vert \col^*(\mathcal{A})\vert = q < m$. With the same notations as in (Subcase 2.1), let 
$\col^*(\mathcal{A}) = \{\col_{h_1}(\mathcal{A}), \ldots, \col_{h_q}(\mathcal{A}), \, 1\le h_1< \cdots < h_q \le m\}$. Setting
$\col_{h_p}(\mathcal{A})= (a_{1,h_p}, a_{2,h_p},\ldots, a_{r,h_p})$ and $\mathcal{C}_{h_p} = \{(k_i, \ell_i) \,:\, a_{i,h_p}\neq 0, \,\, \mbox{for $i=1, \ldots, r$}\}$, for $p =1, \ldots, q$, let $I_{h_p}$ be the strongly stable ideal of $S$ generated in the corner degrees of $\mathcal{C}_{h_p}$ such that $\C(I_{h_p}) = \mathcal{C}_{h_p}$ and
$\beta_{k_i, k_i+\ell_i}(I_{h_p})=a_{i,h_p}$, for $i=1, \ldots, r$.  Consider the monomial submodule $M= M_1 \oplus M_2$ of $S^m$ defined as follows: 
\begin{enumerate}
\item[-] $M_1 = \oplus_{j=1}^{m-q}I_je_j\subsetneq S^{m-q}$, with $G(I_j)= \mathcal{L}(x_1^{\ell_1-1}, \ldots, x_1^{\ell_1-2}x_{k_r+1})$, for $j=1, \ldots, m-q$;
\item[-]  $M_2 = \oplus_{p=1}^{q}I_{h_p}e_{m-q+p}\subsetneq S^q$.
\end{enumerate}
Due to the structure of $M_1$ and $M_2$, one has that $M$ is a strongly stable submodule of $S^m$ with $C_M = \mathcal{A}$. Note that $M_1$ does not give any contribution to the computation of the extremal Betti numbers of $M$.
\end{proof}

The next example illustrates the above theorem.
\begin{Example} \em Let $\mathcal{C}= \{(5,2), (3,3),(2,5)\}$ be a corner sequence and consider the following three positive integers: 
\[
a_1=3, \quad a_2=7, \quad a_3=4.
\] 
Let $S=K[x_1, \ldots, x_6]$. Set $k_1 = 5$, $k_2 = 3$, $k_3 = 2$ and  $\ell_1 = 2$, $\ell_2 = 3$, $\ell_3 = 5$. We want to construct a strongly stable submodule $M \subsetneq S^3$ with $\C(M)= \mathcal{C}$ and such that $\beta_{k_1, k_1+\ell_1}(M) = a_1$, $\beta_{k_2, k_2+\ell_2}(M) = a_2$, $\beta_{k_3, k_3+\ell_3}(M) = a_3$. 
 
An admissible choice for the sequences of non-negative integers $(a_{i,1},a_{i,2},a_{i,3})$ ($i=1,2,3$) such that 
$a_{i,1}+a_{i,2}+a_{i,3}=a_i$ is the following one:
\begin{equation}\label{terne}
(1,0,2), \quad (3,4,0), \quad (1,2,1).
\end{equation}
Either the choice for the integers $a_i$, or the choice for the triplets $(a_{i,1},a_{i,2},a_{i,3})$ ($i=1,2,3$), are forced by Remark \ref{value} together with Propositions \ref{pro:twopairs}, \ref{pro:seven} and \cite[Theorem 3.1]{CU2}.

Indeed, if one considers the pairs $(5,2), (3,3),(2,5)$, then there exists a strongly stable ideal $I$ of $S$ with extremal Betti numbers $\beta_{5, 5+2}(I) = b_1$, $\beta_{3, 3+3}(I) = b_2$, $\beta_{2, 2+5}(I) = b_3$, if and only if $1\le b_1\le \vert A_1\setminus \LexShad(A_0)\vert$, $1\le b_2 \le \vert A_2\setminus \LexShad(A_1)\vert$ and $1\le b_3 \le \vert A_3\setminus \LexShad^2(A_2)\vert$ (Proposition \ref{pro:seven}), where
$A_1 = \{u \in A(5,2) \,: \,u \ge x_1x_6\}$, $A_2 = \{u \in A(3,3)\,:\,u\ge x_2x_4^2\}$, and $A_3 = \{u \in A(2,5)\,:\,u\ge x_3^5\}$. 
If one considers the pairs $(3,3),(2,5)$, then there exists a strongly stable ideal $I$ of $S$ with extremal Betti numbers $\beta_{3, 3+3}(I) = c_1$, $\beta_{2, 2+5}(I) = c_2$, if and only if $1\le c_1\le \vert A_1\setminus \LexShad(A_0)\vert$ and $1\le c_2 \le \vert A_2\setminus \LexShad^2(A_1)\vert$ (\cite[Theorem 3.1]{CU2}), where
$A_1 = \{u \in A(3,3)\,:\,u\ge x_2x_4^2\}$, and $A_2 = \{u \in A(2,5)\,:\,u\ge x_3^5\}$. 
Finally, if one considers the pairs $(5,2),(2,5)$, then there exists a strongly stable ideal $I$ of $S$ with extremal Betti numbers $\beta_{5, 5+2}(I) = d_1$, $\beta_{2, 2+5}(I) = d_2$, if and only if $1\le d_1\le \vert A_1\setminus \LexShad(A_0)\vert$ and $1\le d_2 \le \vert A_2\setminus \LexShad^3(A_1)\vert$ (Proposition  \ref{pro:twopairs}), where
$A_1 = \{u \in A(5,2) \,: \,u \ge x_2x_6\}$, and $A_2 = \{u \in A(2,5)\,:\,u\ge x_3^5\}$. 

The matrix whose rows are the triplets in (\ref{terne}) is the following one:
\[\mathcal{A}= \left(
\begin{array}{lll}
1 & 0 & 2\\
3 & 4 & 0\\
1 & 2 & 1\\
\end{array}
\right).
\]
From Proposition \ref{pro:seven}, there exists a strongly stable ideal $I_1\subsetneq S$ (associated to $\col_1(\mathcal{A})$) such that $\C(I_1) = \mathcal{C}$, $\beta_{k_1, k_1+\ell_1}(I_1) = 1$, $\beta_{k_2, k_2+\ell_2}(I_1)= 3$, $\beta_{k_3, k_3+\ell_3}(I_1) = 1$; from \cite[Theorem 3.1]{CU2}, there exists a strongly  stable ideal $I_2\subsetneq S$ (associated to $\col_2(\mathcal{A})$) such that $\C(I_2) =\{(3,3), (2,5)\}$, $\beta_{k_2, k_2+\ell_2}(I_2)= 4$, $\beta_{k_3, k_3+\ell_3}(I_2) = 2$ and from Proposition \ref{pro:twopairs}, there exists a strongly  stable ideal $I_3\subsetneq S$ (associated to $\col_3(\mathcal{A})$) such that $\C(I_3) =\{(5,2), (2,5)\}$,  $\beta_{k_1, k_1+\ell_1}(I_3)= 2$, $\beta_{k_3, k_3+\ell_3}(I_3) = 1$. More precisely:
\[\begin{array}{ll}
I_1 = (x_1^2, x_1x_2,x_1x_3,x_1x_4,x_1x_5, x_1x_6, x_2^3, x_2^2x_3,x_2^2x_4,x_2x_3^2, x_2x_3x_4, x_2x_4^2, x_3^5),\\
I_2 = (x_1^3, x_1^2x_2,x_1^2x_3,x_1^2x_4,x_1x_2^2, x_1x_2x_3, x_1x_2x_4,x_1x_3^2, x_1x_3x_4,x_1x_4^2, x_2^5, x_2^4x_3, x_2^3x_3^2),\\
I_3 = (x_1^2, x_1x_2,x_1x_3,x_1x_4,x_1x_5, x_1x_6, x_2^2, x_2x_3,x_2x_4,x_2x_5,x_2x_6, x_3^5),
\end{array}\]
and $M = \oplus_{i=1}^3I_ie_i$ is the wanted strongly stable submodule of $S^3$.

\end{Example}
\section*{Acknowledgments}
The author thanks the anonymous referee for his/her careful reading and useful comments leading to a substantial  improvement of the quality of the paper.

\end{document}